\documentclass[A4paper,reqno,12pt]{amsart}
\usepackage{amssymb}
\usepackage{amsmath, mathrsfs,bm}
\usepackage{enumerate,color}

\makeatletter
\@namedef{subjclassname@2020}{%
  \textup{2020} Mathematics Subject Classification}
\makeatother

\newtheorem{lem}{Lemma}[section]

\newtheorem*{theorem*}{Theorem}
\newtheorem{cor}{Corollary}[section]
\newtheorem*{corollary*}{Corollary}

\theoremstyle{definition}
\newtheorem{rem}[cor]{Remark}

\numberwithin{equation}{section}

\newcommand{\CC}{{\mathbb C}}

\newcommand{\NN}{{\mathbb N}}

\newcommand{\RR}{{\mathbb R}}

\newcommand{\rest}[2]{{{#1}_{\kern-.5pt|{#2}}}}
\newcommand{\co}{{\textup{co}\,}}
\newcommand{\Rel}{{\textup{Re}\,}}
\newcommand{\sigmaeps}{{\sigma_\varepsilon}}
\def\C*{{\sl C*}-algebra}
\def\CA*{{\sl C*}-Algebra}
\def\Cs*{{\sl C*}-subalgebra}

\begin{document}

\title[Invertibility preserving mappings onto finite \textsl{C*-}algebras]{Invertibility preserving mappings onto\\ finite \textsl{C*-}algebras}

\author{Martin Mathieu}
\address{Mathematical Sciences Research Centre, Queen's University Belfast, Bel\-fast BT7 1NN, Northern Ireland}
\email{m.m@qub.ac.uk}
\author{Francois Schulz}
\address{Department of Mathematics and Applied Mathematics, Faculty of Science, University of Johannesburg,
P.O. Box 524, Auckland Park, 2006, South Africa}
\email{francoiss@uj.ac.za}



\begin{abstract}
We prove that every surjective unital linear mapping which preserves invertible elements from a Banach algebra onto a \C*
carrying a faithful tracial state is a Jordan homomorphism thus generalising Aupetit's 1998 result for finite von Neumann algebras.
\end{abstract}

\maketitle

\section{Introduction}\label{sect:intro}

\noindent
A linear mapping $T$ between two unital, complex Banach algebras is said to be \textit{spectrum-preserving\/} if, for every element $a$ in the
domain algebra, its spectrum $\sigma(a)$ coincides with $\sigma(Ta)$. Provided the codomain is semisimple and $T$ is surjective, $T$ must be bounded (a result belonging
to Aupetit \cite[Theorem 5.5.2]{Aup91}). Provided the domain is semisimple too, $T$ is injective;
this follows from Zem\'anek's characterisation of the radical (\cite[Theorem 5.3.1]{Aup91}) as
\[
\sigma(a+x)=\sigma(Ta+Tx)=\sigma(Tx)=\sigma(x)\quad\text{for each $a$ such that $Ta=0$ and every $x$}
\]
which implies that $a$ belongs to the radical, which is zero in the semisimple case. Moreover, $T1=1$, that is, $T$ is \textit{unital}.
As a result, a surjective spectrum-preserving mapping between semisimple Banach algebras is a topological isomorphism and one
naturally wonders if it is also an isomorphism of (some of) the algebraic structure.

A \textit{Jordan homomorphism\/} is a linear mapping $T$ with the property $T(a^2)=(Ta)^2$ for all $a$ in the domain (which is equivalent to
$T(ab+ba)=TaTb+TbTa$ for all $a$ and~$b$). A Jordan isomorphism turns out to be spectrum-preserving, and a lot of work has been invested
to explore to what extent the reverse implication holds. A pleasant survey on the history of this topic is contained in~\cite{Aup98};
see also~\cite{Brits2019}, \cite{Ma09} and~\cite{Ma2013} for related questions.

In~\cite{Aup00}, Aupetit proved that every surjective spectrum-preserving linear mapping between von Neumann algebras is a Jordan isomorphism.
It is not difficult to see that it suffices that one of the algebras is a unital \C* of real rank zero and the other a unital semisimple Banach algebra.
However, the problem remains open for general \C*s. It is also known that the assumption on $T$ can be relaxed to a surjective unital
\textit{invertibility-preserving\/} linear mapping (that is, $\sigma(Ta)\subseteq\sigma(a)$ for all~$a$); the conclusion is then that $T$
is a Jordan homomorphism.

In an earlier paper~\cite{Aup98}, Aupetit had already obtained the same result for finite von Neumann algebras. The main tool in that result
was the Fuglede--Kadison determinant $\Delta$;
see \cite[pp.~105]{Dix1957} for its definition and properties. Its relation to the finite trace $\tau$ is given by
$\Delta(a)=\exp(\tau(\log|a|))$, for every invertible element $a$. In our approach we bypass the determinant and work exclusively with a (faithful) tracial state instead
in order to obtain the following generalisation.
\begin{theorem*}\label{thm:main}
Let $B$ be a unital complex Banach algebra and let $A$ be a unital finite \C*. Let $T\colon B\to A$ be a surjective unital linear mapping
which preserves invertible elements. Then $T$ is a Jordan homomorphism.
\end{theorem*}
We largely follow Aupetit's arguments but, to emphasise the differences, we split up the proof into a series of lemmas in the next section.

\section{Preliminaries}\label{sect:prelims}

\noindent
Let $A$ be a unital \C*. We say that $A$ is \textit{finite\/} if it comes equipped with a \textit{faithful tracial state\/},
that is, a linear functional $\tau$ such that $\tau(1)=1=\|1\|$, $\tau(ab)=\tau(ba)$ for all $a,b\in A$ and
$\tau(a^*a)=0$ implies $a=0$. Such a functional is necessarily positive and bounded.

We denote the set of all \textit{states\/} of $A$ (positive linear functionals of norm~$1$) by~$S$ and by $Sp$ the subset
of all \textit{spectral states\/} $f$ of~$A$, that is, $f\in S$ and $|f(x)|\leq\rho(x)$ for every $x\in A$, where
$\rho(x)$ denotes the spectral radius of~$x$. It is known (\cite[Theorem~4 in~\S13]{BonDun}) that every $f\in Sp$
has the trace property, that is, $f(ab)=f(ba)$ for all $a,b\in A$, and that $f(a)\in\co\sigma(a)$, the convex hull of the spectrum
$\sigma(a)$ of $a$, for each $a\in A$; see, \cite[Lemma~2 in~\S13]{BonDun} or \cite[Lemma 4.1.15]{Aup91}.

Conversely, every tracial state $\tau$ belongs to $Sp$ as follows from the subsequent argument.
For $a\in A$, denote by $V(a)=\{f(a)\mid f\in S\}$ its (\textit{algebra\/}) \textit{numerical range} \cite{BonDun}.
As is shown in \cite[Lemma]{Berb69}, and attributed to \cite[\S2]{Hild},
$\co\sigma(a)=\bigcap_{b\in G(A)}V(bab^{-1})$, where $G(A)$ stands for the group of invertible elements in~$A$.
Clearly, $\tau(a)$ belongs to the right hand side of the above identity and hence, $\tau\in Sp$.
(Compare also \cite{Mat2003}.)

\begin{lem}\label{lem:aupetit-1.11}
Let $A$ be a unital \C* with faithful tracial state~$\tau$. Suppose that $g\colon\CC\to A$ is an entire function with values in $G(A)$.
Then the mapping $g_\tau\colon\CC\to\RR$, $g_\tau(\lambda)=\tau(\log(|g(\lambda)|)$ is harmonic.
\end{lem}
\begin{proof}
The argument in the proof of Th\'eor\`eme~1.11 in~\cite{Aup98}, which is already entirely formulated in terms of the trace,
takes over verbatim.
\end{proof}
In the following, $T$ will denote a surjective unital linear mapping defined on a (complex, unital) Banach algebra $B$ with values
in a finite unital \C*~$A$. We will assume that $T$ \textit{preserves invertible elements} so that $T\mkern.5mu G(B)\subseteq G(A)$.
It follows from \cite[Theorem 5.5.2]{Aup91} that $T$ is bounded.

Fix $a,b\in B$ and define
\[
g\colon\CC\times\CC\longrightarrow G(A),\quad g(\lambda,\mu)=T(e^{\lambda a}e^{\mu b})e^{-\lambda Ta}e^{-\mu Tb}.
\]
Then $g$ is a separately entire function. Its series expansion reads as follows
\begin{equation}\label{eq:series}
\begin{split}
g(\lambda,\mu) = 1 &+\frac{\lambda^2}{2}\bigl(T(a^2)-(Ta)^2\bigr) + \frac{\mu^2}{2}\bigl(T(b^2)-(Tb)^2\bigr)\\
                   &+\lambda\mu\bigl(T(ab)-TbTa\bigr)+\frac{\lambda^3}{6}\bigl(T(a^3)+2(Ta)^3-3T(a^2)Ta\bigr)\\
                   &+\frac{\lambda^2\mu}{2}\bigl(T(a^2b)+(Ta)^2Tb+Tb(Ta)^2-T(a^2)Tb-2T(ab)Ta\bigr)\\
                   &+\frac{\lambda\mu^2}{2}\bigl(T(ab^2)+2TbTaTb-2T(ab)Tb-T(b^2)Ta\bigr)\\
                   &+\frac{\mu^3}{6}\bigl(T(b^3)+2(Tb)^3-3T(b^2)Tb\bigr) +\textup{remainder}
\end{split}
\end{equation}
where the remainder only contains terms of degree~$4$ or higher in $\lambda$ and~$\mu$; we will put it to good use in
the proof of the main theorem.

By Lemma~\ref{lem:aupetit-1.11}, the function
\[
g_\tau\colon\CC\times\CC\longrightarrow\RR,\quad g_\tau(\lambda,\mu)=\tau(\log(|g(\lambda,\mu)|))
\]
is separately harmonic in $\lambda$ and~$\mu$ and thus there exists a separately entire function $h(\lambda,\mu)$ such that
$\Rel h(\lambda,\mu)=g_\tau(\lambda,\mu)$ for all $\lambda,\mu\in\CC$.

\smallskip
The next step will be to establish the following three lemmas; for their proofs, see Section~\ref{sect:proofs}.
\begin{lem}\label{lem:bounded}
For all $\lambda,\mu\in\CC$, we have $\,e^{\mkern.8mu g_\tau(\lambda,\mu)}\leq\|g(\lambda,\mu)\|$.
\end{lem}

\begin{lem}\label{lem:real-part}
With the above notation and caveats, let $g^*(\lambda,\mu)$ stand for $(g(\lambda,\mu))^*$.
Then there exists $r>0$ such that, for all $\lambda,\mu\in\CC$ with $|\lambda|,|\mu|<r$, we have
\begin{equation}\label{eq:trace-log-expansion}
2\,\Rel h(\lambda,\mu)=\tau\bigl(\log(g^*(\lambda,\mu)g(\lambda,\mu)\bigl)=-\sum_{k=1}^\infty\frac 1k\tau\bigl((1-g^*(\lambda,\mu)g(\lambda,\mu))^k\bigr).
\end{equation}
\end{lem}

\begin{lem}\label{lem:identical-zero}
For all $\lambda,\mu$ in a neighbourhood of zero, $\tau\bigl(\log(g^*(\lambda,\mu)g(\lambda,\mu)\bigl)=0$.
\end{lem}

\section{Proofs of the Lemmas and the Main Theorem}\label{sect:proofs}

\noindent
The argument of the first lemma differs from~\cite{Aup98} in that we cannot make use of the determinant
in order to locate the appropriate values in the convex hull of the spectrum.
\begin{proof}[Proof of Lemma~\ref{lem:bounded}]
As we observed above, $\tau(\log(|g(\lambda,\mu)|))\in\co\sigma(\log(|g(\lambda,\mu)|))$ for all $\lambda,\mu\in\CC$.
Since the spectrum of $\log(|g(\lambda,\mu)|)$ is contained in $\RR$, it follows that $\co\sigma(\log(|g(\lambda,\mu)|))=[s,t]$
for some $s,t\in\sigma(\log(|g(\lambda,\mu)|))$ with $s\leq t$.
Since the exponential function is strictly increasing, the Spectral Mapping Theorem implies that
\begin{equation*}
e^{\mkern.8mu g_\tau(\lambda,\mu)}\in[e^s,e^t]=\co\sigma(e^{\log(|g(\lambda,\mu)|)})
                                              =\co\sigma(|g(\lambda,\mu)|).
\end{equation*}
As a result, $0<e^{\mkern.8mu g_\tau(\lambda,\mu)}\leq\rho(|g(\lambda,\mu)|)$ and therefore,
\begin{equation*}
\begin{split}
e^{\mkern.8mu 2\,g_\tau(\lambda,\mu)} &\leq \rho(|g(\lambda,\mu)|)^2=\rho(|g(\lambda,\mu)|^2)\\
                                      &=\rho(g^*(\lambda,\mu)g(\lambda,\mu))=\|g^*(\lambda,\mu)g(\lambda,\mu)\|\\
                                      &=\|g(\lambda,\mu)\|^2
\end{split}
\end{equation*}
as claimed.
\end{proof}
The next argument is rather straightforward.
\begin{proof}[Proof of Lemma~\ref{lem:real-part}]
As $g(0,0)=T1=1$, by continuity, there is $r>0$ such that, for all $\lambda,\mu$ with $|\lambda|,|\mu|<r$, we have
$\|1-g^*(\lambda,\mu)g(\lambda,\mu)\|<1$. The series expansion of the logarithm thus yields
\begin{equation}\label{eq:log-expansion}
-\log(g^*(\lambda,\mu)g(\lambda,\mu))=\sum_{k=1}^\infty\frac 1k(1-g^*(\lambda,\mu)g(\lambda,\mu))^k.
\end{equation}
The definition of $g_\tau$ entails that
\begin{equation}
2\,\Rel h(\lambda,\mu)=2\,\tau(\log(|g(\lambda,\mu)|))=\tau(\log(|g(\lambda,\mu)|^2))=\tau\bigl(\log(g^*(\lambda,\mu)g(\lambda,\mu)\bigl).
\end{equation}
Combining these two identities gives the claim.
\end{proof}
The proof of the third lemma follows exactly Aupetit's arguments. (There appears to be some misprint at the bottom of page~61 and top of page~62
of~\cite{Aup98}.)
\begin{proof}[Proof of Lemma~\ref{lem:identical-zero}]
For all $\lambda,\mu\in\CC$, we have
\[
\bigl|e^{\mkern.8mu h(\lambda,\mu)}\bigr| = e^{\mkern.8mu\Rel h(\lambda,\mu)}=e^{\mkern.8mu g_\tau(\lambda,\mu)}\leq\|g(\lambda,\mu)\|
\]
by Lemma~\ref{lem:bounded}. Since
\[
\|g(\lambda,\mu)\|\leq\|T\|\,e^{\mkern.08mu|\lambda|(\|a\|+\|Ta\|)+|\mu|(\|b\|+\|Tb\|)}
\]
it follows that $e^{\mkern.8mu h(\lambda,\mu)}=e^{\alpha\lambda+\beta\mu+\gamma}$ for suitable $\alpha,\beta,\gamma\in\CC$
(\cite[Lemma~3.2]{Aup96}). As $g_\tau(0,0)=0$ we have $|e^\gamma|=1$, thus we may assume that $\gamma=0$ (since we need only the real part of~$\gamma$).
Therefore, $2\,\Rel h(\lambda,\mu)=\alpha\lambda+\beta\mu+\bar\alpha\bar\lambda+\bar\beta\bar\mu$.
From Lemma~\ref{lem:real-part} we obtain
\[
\alpha\lambda+\beta\mu+\bar\alpha\bar\lambda+\bar\beta\bar\mu=-\sum_{k=1}^\infty\frac 1k\tau\bigl((1-g^*(\lambda,\mu)g(\lambda,\mu))^k\bigr)
\]
for all $\lambda,\mu$ such that $|\lambda|,|\mu|<r$ for suitable $r>0$.
The series expansion of $g(\lambda,\mu)$ in \eqref{eq:series} does not contain any powers of $\lambda$ or $\mu$ of first order,
hence the series expansion in~\eqref{eq:log-expansion} cannot either. This entails that both $\alpha$ and $\beta$ are equal to zero.

It now follows from~\eqref{eq:trace-log-expansion} that $\tau\bigl(\log(g^*(\lambda,\mu)g(\lambda,\mu)\bigl)=0$.
\end{proof}
We now have all the tools to prove our main theorem by adapting the arguments in Theorem 1.12 of~\cite{Aup98} to our situation.
\begin{proof}[Proof of the Theorem]
Set $f(\lambda,\mu)=\sum_{k=1}^\infty\frac 1k\tau\bigl((1-g^*(\lambda,\mu)g(\lambda,\mu))^k\bigr)$
for all $\lambda,\mu$ such that $|\lambda|,|\mu|<r$ for suitable $r>0$ (given by Lemma~\ref{lem:real-part}).
By Lemma~\ref{lem:identical-zero}, $f=0$ and thus
\[
\frac{\partial^2}{\partial\lambda\partial\mu}f(0,0)=0=\frac{\partial^3}{\partial\lambda^2\partial\mu}f(0,0).
\]
Using these identities after substituting in the series expansion~\eqref{eq:series} into the log-series we find
\begin{equation}\label{eq:first}
\tau\big(T(ab)-TaTb\bigr)=0
\end{equation}
and
\begin{equation}\label{eq:second}
\tau\bigl(T(a^2b)+(Ta)^2Tb+Tb(Ta)^2-T(a^2)Tb-2\,T(ab)Ta\bigr)=0
\end{equation}
for all $a,b\in B$. From~\eqref{eq:first} we obtain
\[
\tau(T(a^2b))=\tau(T(a^2)Tb)
\]
and
\[
\tau(T(a^2b))=\tau(T(a(ab)))=\tau(TaT(ab))
\]
so that~\eqref{eq:second} reduces to
\[
\tau((Ta)^2Tb)=\tau(TaT(ab)),
\]
using the trace property.
It follows that
\[
\tau((Ta)^2Tb)=\tau(T(a^2)Tb).
\]
Since $T$ is surjective we may choose $b\in B$ such that
$Tb=((Ta)^2-T(a^2))^*$ wherefore the last identity yields, for each $a\in B$, that
\[
\tau\bigl(((Ta)^2-T(a^2))((Ta)^2-T(a^2))^*\bigr)=0.
\]
The faithfulness of $\tau$ implies that $T$ is a Jordan homomorphism.
\end{proof}

\section{Conclusions}\label{sect:consls}

\noindent
In this section, we collect together some consequences and sharpening of our main theorem.
We also relate it to open problems of a similar nature.

Suppose $T$ is a surjective linear mapping between two semisimple unital Banach algebras which preserves the spectrum of each element.
Then $T$ is injective (as explained in the Introduction) and $T1=1$. The latter follows, for example, from
\begin{equation*}
\sigma((T1-1)+Tx)=\sigma(T1+Tx)-1=\sigma(1+x)-1=\sigma(x)=\sigma(Tx)
\end{equation*}
and the surjectivity of $T$ which entails that $\sigma((T1-1)+y)=\sigma(y)$ for all $y$ in the codomain. Thus, by Zem\'anek's
characterisation of the radical, $T1-1=0$.

As a result, we have a symmetric situation and can apply the Theorem to either $T$ or its inverse to obtain the following consequence.
\begin{cor}\label{cor:spectrum-preserv}
Let $T$ be a surjective spectrum-preserving linear mapping between two semisimple unital Banach algebras. If either of them is a unital \C*
equipped with a faithful tracial state then $T$ is a Jordan isomorphism.
\end{cor}
This is another contribution to a longstanding, still open problem by Kaplansky who asked in 1970 whether the above statement holds
without any further assumptions on the Banach algebras. For further references, see~\cite{Aup98} and~\cite{Harris2001}.

All the steps in the proof of the Theorem but the very last one can be performed for each individual tracial state on a unital \C*.
Therefore, the assumption can be relaxed to the existence of a faithful family of tracial states, that is, a family $\{\tau_i\mid i\in I\}$
of tracial states $\tau_i$ such that $\tau_i(a^*a)=0$ for all $i\in I$ implies $a=0$.

In particular, since any tracial state on a simple unital \C* is faithful we obtain the following result.
\begin{cor}\label{cor:simple-case}
Let $T\colon B\to A$ be a surjective unital invertibility-preserving linear mapping into a simple unital \C*~$A$ which carries a tracial state.
Then $T$ is a Jordan homomorphism.
\end{cor}
\begin{rem}\label{rem:terminology}
Our terminology of a ``finite''\C* is not quite standard. In~\cite[III.1.3.1]{Black}, a unital \C* $A$ is called \textit{finite\/} if the identity of $A$ is a finite projection;
that is, there is no proper subprojection which is Murray--von Neumann equivalent to~$1$. Every unital \C* with a faithful tracial state is finite in this sense
but the converse fails in general (though it holds for stably finite exact \C*s). We prefer here a definition that does not make reference to any projections.
\end{rem}
We can also strengthen our main theorem in a different direction. Let $G_1(B)$ denote the \textit{principal component of\/} $G(B)$, where $B$ is a unital Banach
algebra. It is known, see, e.g., \cite[Theorem 3.3.7]{Aup91}, that $G_1(B)=\{e^{x_1}\cdots e^{x_n}\mid x_i\in B,\,n\in\NN\}$.
The associated \textit{exponential spectrum\/} of $x\in B$ is
\[
\sigmaeps(x)=\{\lambda\in\CC\mid\lambda-x\notin G_1(B)\}.
\]
In certain situations it is more natural and expedient to consider the exponential spectrum instead of the smaller spectrum, see, e.g.,~\cite[Theorem 3.3.8]{Aup91}.
From the proof of our main result we see that it suffices that the mapping $T$ sends the product of any two exponentials in $B$ onto an invertible element in~$A$.
This gives the following corollary.
\begin{cor}\label{cor:exp-spectrum}
Let $B$ be a unital complex Banach algebra and let $A$ be a unital finite \C*. Let $T\colon B\to A$ be a surjective unital linear mapping
such that $TG_1(B)\subseteq G(A)$. Then $T$ is a Jordan homomorphism.
\end{cor}
A \textit{spectral isometry\/} between two Banach algebras $A$ and $B$ is a linear mapping $S$ such that $\rho(Sx)=\rho(x)$ for all $x\in A$.
Clearly, every spectrum-preserving mapping is a spectral isometry and so is every Jordan isomorphism.
A conjecture related to Kaplansky's problem mentioned above states that every unital surjective spectral isometry between two \C*s is a Jordan isomorphism.
This conjecture has been confirmed in many cases, see, e.g., \cite{Ma09, Ma2013}, but is open in all generality.
Notably it was verified in \cite{Ma2020} if $A$ is a unital \C* of real rank zero and without tracial states.
The above Corollary~\ref{cor:spectrum-preserv} is thus a step forward in the direction of confirming the general conjecture.


\noindent
\textbf{Acknowledgements.} The research for this paper was completed while the second-named author was visiting Queen's University Belfast.
He would like to thank both Professor Martin Mathieu and the Mathematical Sciences Research Centre at Queen's University Belfast for their hospitality,
and the National Research Foundation of South Africa for their financial support (NRF Grant Number: 129692).

\end{document}